\documentclass[11pt, a4paper]{amsart}

\usepackage{amsmath, amssymb, amsthm}
\usepackage[english]{babel}
\usepackage[all]{xy}
\usepackage{graphicx}
\usepackage{tikz}
\usetikzlibrary{matrix}
\usepackage{verbatim}
\usepackage{enumerate}
\usepackage{url}
\usepackage{subfig}
\usepackage{bm}
\usepackage{xcolor}

\usetikzlibrary{decorations.pathreplacing,decorations.markings,positioning}
\tikzset{
    middlearrow/.style n args={3}{
        draw,
        decoration={
            markings,
            mark=at position 0.5 with {
                \arrow[scale=2]{#1};
                \path[#2] node {$#3$};
            },
        },
        postaction=decorate
    }
}

\newtheorem{thm}{Theorem}[section]
\newtheorem{lem}[thm]{Lemma}

\newtheorem{prop}[thm]{Proposition}
\newtheorem{cor}[thm]{Corollary}
\theoremstyle{definition}
\newtheorem{defn}[thm]{Definition}
\newtheorem{ex}[thm]{Example}
\theoremstyle{remark}
\newtheorem{rem}[thm]{Remark}

\newcommand{\note}[1]{}

\newcommand{\HC}[2]{[#1\to#2]}

\newcommand{\Cyc}{\textsf{\textbf{C}}}
\newcommand{\Net}{\textsf{\textbf{N}}}

\newcommand{\Wu}{W^\textrm{u}}
\newcommand{\Ws}{W^\textrm{s}}

\newcommand{\Shft}{\Sigma}

\newcommand{\dd}{\textnormal{d}}

\newcommand{\diag}{\textnormal{diag}}

\usepackage{xcolor}
\definecolor{colCB}{rgb}{0,0,.7}
\definecolor{colAL}{rgb}{.7,0,0}

\newcommand{\F}{\mathcal{F}}

\newcommand{\R}{\mathbb{R}}
\newcommand{\Rn}{\R^d}

\newcommand{\N}{\mathbb{N}}

\newcommand{\norm}[1]{\left\|#1\right\|}

\newcommand{\rset}[2]{\left\lbrace\, #1\,\left|\;#2\right.\right\rbrace}

\newcommand{\set}[2]{\rset{#1}{#2}}
\newcommand{\tset}[2]{\big\lbrace #1\,\big|\;#2\big\rbrace}
\newcommand{\sset}[1]{\left\lbrace #1\right\rbrace}

\begin{document}


\title[Towards heteroclinic networks in aperiodic ways]{How many points converge to a heteroclinic network in an aperiodic way?}
\author[Christian Bick and Alexander Lohse]{Christian Bick${}^\blacktriangledown$ and Alexander Lohse${}^\blacktriangle$}%

\address{${}^\blacktriangledown$Department of Mathematics, Vrije Universiteit Amsterdam, DE Boelelaan 1111, Amsterdam, The Netherlands;
Institute for Advanced Study, Technical University of Munich, Lichtenbergstr 2, 85748 Garching, Germany; 
Department of Mathematics, University of Exeter, Exeter EX4 4QF, United Kingdom; 
and Mathematical Institute,
University of Oxford, Oxford OX2 6GG, United Kingdom}
\address{${}^\blacktriangle$University of Hamburg, Department of Mathematics, Bundesstra{\ss}e 55, 20146 Hamburg, Germany}
\date{\today}

\maketitle

\begin{abstract}
Homoclinic and heteroclinic connections can form cycles and networks in phase space, which organize the global phenomena in dynamical systems.
On the one hand, stability notions for (omni)cycles give insight into how many initial conditions approach the network along a single given (omni)cycle.
On the other hand, the term \emph{switching} is used to describe situations where there are trajectories that follow any possible sequence of heteroclinic connections along the network.
Here we give a notion of asymptotic stability for general sequences along a network of homoclinic and heteroclinic connections.
We show that there cannot be uncountably many aperiodic sequences that attract a set with nontrivial measure.
Finally, we discuss examples where one may or may not expect aperiodic convergence towards a network and conclude with some open questions.
\end{abstract}

\newcommand{\HCi}[2]{\HC{\xi_{#1}}{\xi_{#2}}}

\section{Introduction}
\noindent
Homoclinic and heteroclinic structures are prototypical examples of invariant sets in phase space that describe global dynamical phenomena.
They are not only of intrinsic mathematical interest~\cite{Ashwin2017} but have also been employed to model and understand metastable dynamics in applications, including neuroscience~\cite{Rabinovich2001,Rabinovich2006,Rabinovich2012,Creaser2019}, coupled oscillators~\cite{Bick2018}, game theory~\cite{Postlethwaite2022,Castro2021}, and computation~\cite{SchittlerNeves2012}.
Typically, invariant sets are physically most relevant if they attract a set of initial conditions that is nonvanishing with respect to the Lebesgue measure---in this case the invariant set has been called \emph{fragmentarily asymptotically stable (f.a.s.)}~\cite{Podvigina2012}.

To be more specific, consider a dynamical system on~$\Rn$ defined by a differential equation $\dot x = f(x)$, $x\in\Rn$, which induces a sufficiently smooth flow~$\phi_t$. 
For a hyperbolic equilibrium~$\xi$ let~$\Wu(\xi)$, $\Ws(\xi)$ denote the usual unstable/stable invariant manifolds, respectively.
A \emph{cycle~$\Cyc$ in phase space} of length~$L$ consists of finitely many distinct hyperbolic equilibria~$\xi_{1}, \dotsc, \xi_{L}\in\R^d$ and connecting trajectories $\gamma_{l}\subset \Wu(\xi_{l})\cap \Ws(\xi_{l+1})$, where indices~$\l$ are taken modulo~$L$. 
A cycle is \emph{homoclinic} if $L=1$ and \emph{heteroclinic} if $L>1$. 
We identify a cycle~$\Cyc$ with its invariant set, that is, $\Cyc = \bigcup_{l=1}^L (\sset{\xi_{l}}\cup\gamma_{l})$.


Testable conditions for a cycle to be f.a.s.\ can be found under reasonable assumptions on the return maps which are used to approximate the dynamics of trajectories near the cycle: Introducing logarithmic coordinates to write the local dynamics near equilibria in terms of \emph{transition matrices}~\cite{Field1991,Krupa2004}, properties of the eigenvectors and eigenvalues of finite products of these transition matrices determine whether a cycle is f.a.s.~\cite{Podvigina2012,Garrido-da-Silva2016}.

Multiple cycles can make up larger structures in phase space:
A \emph{network~$\Net$ in phase space} is the connected finite union of two or more distinct cycles\footnote{Note that we do not preclude the existence of infinitely many connecting trajectories between equilibria~$\xi_p$ and $\xi_{p'}$, for example, if $\dim(\Wu(\xi_{p})\cap \Ws(\xi_{p'}))>1$. We rather restrict attention to a finite collection of connecting trajectories.}, possibly of different lengths. 
Its nodes are hyperbolic equilibria~$\xi_p$, $p\in\sset{1, \dotsc, P}$ connected by trajectories~$\gamma_q$, $q\in\sset{1, \dotsc, Q}$.
In particular, if $\Xi = \sset{\xi_p}$ is the set of nodes and $\alpha, \omega$ denote the usual limit sets, we have for any $x\in\gamma_q$ that $\alpha(x), \omega(x)\in\Xi$.
The network as an invariant set, that is, $\Net = \bigcup_{p=1}^{P}\sset{\xi_p}\cup\bigcup_{q=1}^{Q}\gamma_q $, may also be~f.a.s.

The main question we address here is \emph{how} trajectories approach a network~$\Net$ that is f.a.s. 
In particular, we are interested in how many initial conditions approach~$\Net$ following a given, potentially aperiodic, sequence of connections.
These questions are motivated from two perspectives.

First, one may consider the stability of constituent cycles.
For example, Postlethwaite and Dawes~\cite{Postlethwaite_2005} observed ``irregular cycling'' dynamics close to a heteroclinic network and obtained restrictions for periodic cycling.
More recently, Podvigina~\cite{Podvigina2021} extended the notion of stability of cycles to \emph{omnicycles}---periodic concatenations of cycles---and gave criteria that make it possible to compute the stability properties explicitly in the same way as for cycles.
In other words, omnicycles can be seen as periodic sequences of transitions along connections in a network, which yields a notion of stability for such periodic sequences (beyond any associated invariant sets).

Second, there is the concept of \emph{switching} near a homo-/heteroclinic network~\cite{Aguiar2005,doi:10.1080/14689361003769770,Rodrigues2015a}. 
Roughly speaking, switching means that for any arbitrarily small neighborhood of the network and for any given sequence to follow the connections in the network, there is a trajectory within this neighborhood that does exactly that.
In some cases it is possible to prove the existence of switching by showing that there is an invariant set where the dynamics are conjugated to a full shift on a set of symbols suitably chosen to represent all possible paths around the network~\cite{Aguiar2005}.
In such examples any sequence of network connections may be followed, but the trajectories doing so are not necessarily attracted to the network.

Consequently, there is a gap between stability of (omni)cycles and switching.
On the one hand, numerical simulations indicate that there may be trajectories that approach a network in an aperiodic way~\cite{Postlethwaite_2005,Postlethwaite2022, Podvigina2021}---but it is not clear whether for any aperiodic path there are trajectories following it and, if so, how many initial conditions follow any given path.
On the other hand, there are well-studied examples where there are obstructions to switching, such as the Kirk--Silber network~\cite{VKirk_1994} and more general networks~\cite{Castro_2023}: 
There are parameter values for which trajectories may wind around one cycle arbitrarily often but as soon as they transition to another cycle, they will stay there indefinitely.

In this note we introduce and discuss stability properties of arbitrary sequences of transitions between equilibria in a network~$\Net$ with a particular focus on aperiodic sequences.
Section~\ref{sec:Sequences} introduces basic terminology on networks and sequences as well as the concept of switching.
In Section~\ref{sec:Stability}, we then define a notion of fragmentary asymptotic stability for sequences of connections in a network, whether periodic or aperiodic. 
We show that our definition is compatible with the existing notions of fragmentary asymptotic stability for invariant sets.
In Section~\ref{sec:Aperiodicity} we discuss some consequences for aperiodic sequences. 
In particular, we show that if there is one f.a.s.\ aperiodic sequence, then there are countably infinitely many f.a.s.\ aperiodic sequences (Proposition~\ref{prop:CountablyMany}), but there cannot be uncountably many f.a.s.\ aperiodic sequences (Proposition~\ref{prop:NotTooMany}).
Finally, in Section~\ref{sec:Examples} we discuss concrete examples where aperiodicity does and does not play a role. We finish with a discussion and some open questions in Section~\ref{sec:discussion}.

\section{Networks, sequences, and switching}
\label{sec:Sequences}

\newcommand{\qb}{\mathbf{q}}
\newcommand{\pb}{\mathbf{p}}

\subsection{Sequences on a network}

Consider a network~$\Net \subset \Rn$ in phase space with $P$~equilibria~$\Xi = \sset{\xi_p}$ and $Q$~heteroclinic~connections~$\gamma_q$. 
For any~$\gamma_q$ and $x\in\gamma_q$ write $\alpha[\gamma_q] := \alpha(x)$, $\omega[\gamma_q] := \omega(x)$, so that $\gamma_q$ is a connection between the equilibria $\alpha[\gamma_q]$ and $\omega[\gamma_q]$. Itineraries along the network~$\Net$ can be encoded by sequences of transitions:

\begin{defn}\label{def:TransitionSeq}
An (infinite) sequence $\qb = q_1q_2q_3\ldots$ with $q_k\in\sset{1,\dotsc,Q}$ is a \emph{sequence on~$\Net$} if $\omega[\gamma_{q_k}] = \alpha[\gamma_{q_{k+1}}]$ for all $k\in\N$.
\end{defn}

\begin{rem}
In some cases itineracy in terms of \emph{sequences of transitions~$\qb$} as in Definition~\ref{def:TransitionSeq} can be identified with \emph{sequences of equilibria}:
Suppose that for each pair~$(\xi_p, \xi_{p'})$ there is at most one connecting trajectory~$\gamma_q$ in~$\Net$ with $\alpha[\gamma_q] = \xi_p$ and $\alpha[\gamma_q] = \xi_{p'}$.
Then one may identify the sequence $\qb = q_1q_2\ldots$ of connections with a sequence of equilibria $\pb=p_1p_2\ldots\in\sset{1,\dotsc,P}^\N$ of equilibria such that $\xi_{p_k}=\alpha[\gamma_{q_k}]$, $\xi_{p_{k+1}} = \omega[\gamma_{q_k}]$.
\end{rem}

Each sequence defines an associated invariant set that is a subset of the network. For a sequence~$\qb$ define
\begin{equation}\label{eq:Gam}
\Gamma(\qb) := \bigcup_{j\in\N} \left(\sset{\alpha[\gamma_{q_j}]}\cup\gamma_{q_j}\cup\sset{\omega[\gamma_{q_j}]}\right)\subset\Net,
\end{equation}
an invariant set containing all connections in~$\qb$ as well as the equilibria that are their endpoints. 

\begin{defn}
A heteroclinic sequence $\qb = q_1q_2q_3\ldots$ is \emph{preperiodic} with period $K>0$ if there is a $K_0 \geq 0$ such that $q_{k+K}=q_k$ for all $k>K_0$. It is \emph{periodic} if it is preperiodic with $K_0=0$. The sequence is \emph{aperiodic} if it is not preperiodic.
\end{defn}

A (homoclinic or heteroclinic) cycle as an invariant set~$\Cyc$ can be associated with a periodic sequence~$\qb$. 
In this case $\Gamma(\qb)=\Cyc$ is topologically a circle, because our definition of a cycle restricts attention to a single connection between a given pair of equilibria. 
Similarly, omnicycles are associated with periodic sequences. 
Conversely, for any periodic sequence~$\qb$ the invariant set~$\Gamma(\qb)$ is, in general, topologically a union of circles.

Moving along a sequence now corresponds to a shift operation on a sequence of~$Q$ symbols for each of the connections in the network~$\Net$.

\begin{defn}
The \emph{shift operator} $\Shft: \sset{1, \dotsc, Q}^\N\to \sset{1, \dotsc, Q}^\N$ acts as a one-sided shift on sequences, that is,
\[\Shft(q_1q_2q_3\ldots) := q_2q_3q_4\ldots.\]
\end{defn}

While we talk about sequences on a network for the remainder of the paper, note that this is not necessarily the most efficient coding. For a general network, an optimal coding should take into account what happens at \emph{distribution nodes}\footnote{While \emph{distribution node} is used in~\cite{Castro_2023}, such nodes are called \emph{splitting nodes} in~\cite{AshwinCastroLohse2020}.}, i.e., equilibria that have more than one outgoing heteroclinic connection in the network~$\Net$. We illustrate this with an example.

\begin{ex}\label{ex:KS}
The Kirk--Silber network~\cite{VKirk_1994} is the union of two distinct cycles with one joint connection between four equilibria $\xi_1, \dotsc, \xi_4$; cf.~Figure~\ref{fig:ks}.
If $\HCi{p}{p'}$ denotes the heteroclinic connection from~$\xi_p$ to~$\xi_{p'}$, then the first cycle has connections $\gamma_1 = \HCi{1}{2}$, $\gamma_2=\HCi{2}{3}$, $\gamma_3=\HCi{3}{1}$ and the second cycle has $\gamma_1 = \HCi{1}{2}, \gamma_4=\HCi{2}{4}$, $\gamma_5=\HCi{4}{1}$. The cycles correspond to the periodic sequences $123123\dotsb$ and $145145\dotsb$, respectively.

The only relevant information is what happens at the {distribution node~$\xi_2$}:
All sequences~$\qb$ that start with the joint connection~$\gamma_1$ are concatenations of the symbols $\mathrm{A} := 123$, $\mathrm{B} := 145$, i.e., they can be written compactly as sequences on the symbols $\sset{\mathrm{A},\mathrm{B}}$.
(For sequences of equilibria~$\pb$, the symbols are $\mathrm{A} := 123$, $\mathrm{B} := 124$.)
The shift~$\Shft$ induces a shift operator $\widetilde\Shft:\sset{\mathrm{A},\mathrm{B}}^\N\to\sset{\mathrm{A},\mathrm{B}}^\N$ that corresponds to a full transition around one of the two cycles.

While there are only two heteroclinic cycles in this network, there are infinitely many omnicycles. These correspond to periodic sequences such as $\qb=123145123145\dotsb$, which consists of cycle~$\mathrm{A}$ once followed by~$\mathrm{B}$ once.
\end{ex}

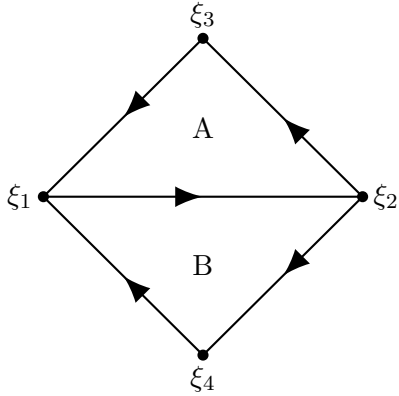
\begin{figure}
\centering
	\begin{tikzpicture}[>=latex,thick,scale=0.7]
	\coordinate (1) at (-3,0);
	\coordinate (2) at (3,0);
	\coordinate (3) at (0,3);
	\coordinate (4) at (0,-3);
	\filldraw[black] (1) circle (2.4pt) node[left] {$\xi_1$};
	\filldraw[black] (2) circle (2.4pt) node[right] {$\xi_2$};
	\filldraw[black] (3) circle (2.4pt) node[above] {$\xi_3$};
	\filldraw[black] (4) circle (2.4pt) node[below] {$\xi_4$};
	\path (1) edge[middlearrow={>}{left}{}] (2);
	\path (2) edge[middlearrow={>}{left}{}] (3);
	\path (3) edge[middlearrow={>}{left}{}] (1);
	\path (2) edge[middlearrow={>}{left}{}] (4);
	\path (4) edge[middlearrow={>}{left}{}] (1);
	\draw (1) to [edge node={node [above] {$\gamma_1$}}] (2);
	\draw (2) to [edge node={node [above right] {$\gamma_2$}}] (3);
	\draw (3) to [edge node={node [above left] {$\gamma_3$}}] (1);
	\draw (2) to [edge node={node [below right] {$\gamma_4$}}] (4);
	\draw (4) to [edge node={node [below left] {$\gamma_5$}}] (1);
	\node at (0,1.3) {$\mathrm{A}$};
	\node at (0,-1.3) {$\mathrm{B}$};
	\end{tikzpicture}
\caption{The Kirk--Silber network.}
\label{fig:ks}
\end{figure}

Finding efficient codings in this way relates to the concept of a skeleton map introduced in~\cite{PeixeRodrigues_2023}.

Finally, we remark that for every network~$\Net$ there are uncountably many sequences on~$\Net$: Even in the simplest case, where~$\Net$ consists of only two cycles (as in the Kirk--Silber network above), the set of itineraries on~$\Net$ corresponds to the set of sequences in two symbols, which is uncountable.

\subsection{Switching near a network}

Sequences may be used to encode the dynamics near a network~$\Net \subset \Rn$.
Let~$U_\Net$ be a neighborhood of the network and let~$U_p$ be a neighborhood of~$\xi_p$. 
For each connection~$\gamma_q$ in~$\Net$ consider a point $\zeta_q\in\gamma_q$ and let~$V_{q}$ be a neighborhood of~$\zeta_q$. 
Assume that all neighborhoods $U_p, V_{q}$ are disjoint.

\begin{defn}[See~\cite{Aguiar2005,Aguiar_2010}]
Consider neighborhoods as defined above. 
A trajectory~$x(t):=\phi_t(x)$ for a given $x \in \Rn$ \emph{follows the sequence $\qb = q_1q_2q_3\ldots$ within~$U_\Net$} if there exist two monotonically increasing sequences of non-negative times~$t_k$ and $t'_k$, $k\in\N$, such that for all $k\in\N$ we have $t_k<t'_k<t_{k+1}$ and 
\begin{enumerate}[(i)]
\item $x(t)\subset U_\Net$ for all $t>t_1$;
\item $x(t_k)\in U_{\alpha[\gamma_{q_k}]}$ and $x(t'_k)\in V_{q_k}$;
\item for all $t\in(t'_k, t'_{k+1})$ we have $x(t)\not\in U_{p'}$ for $\xi_{p'}\neq \omega[\gamma_{q_k}]$.
\end{enumerate}
\end{defn}

We can now identify initial conditions that stay near the network~$\Net$ by following a given sequence~$\qb$; cf.~Figure~\ref{fig:following} for an illustration. For $\delta>0$ let $B_\delta(\Net)$ be a $\delta$-neighborhood of~$\Net$ and let
\begin{align}
\mathcal{S}_\delta({\qb}) &= \set{x\in B_\delta(\Net)}{\phi_t(x) \text{ follows } \qb \text{ with } U_\Net = B_\delta(\Net)}
\end{align}
denote the set of initial conditions that follow the network along~$\qb$ in a $\delta$-neighborhood, or simply the \emph{$\delta$-stable set of the sequence~$\qb$}. We say that a trajectory in~$\mathcal{S}_\delta(\qb)$ \emph{realizes the sequence~$\qb$ in~$B_\delta(\Net)$}.

\begin{figure}
\includegraphics[width=\textwidth]{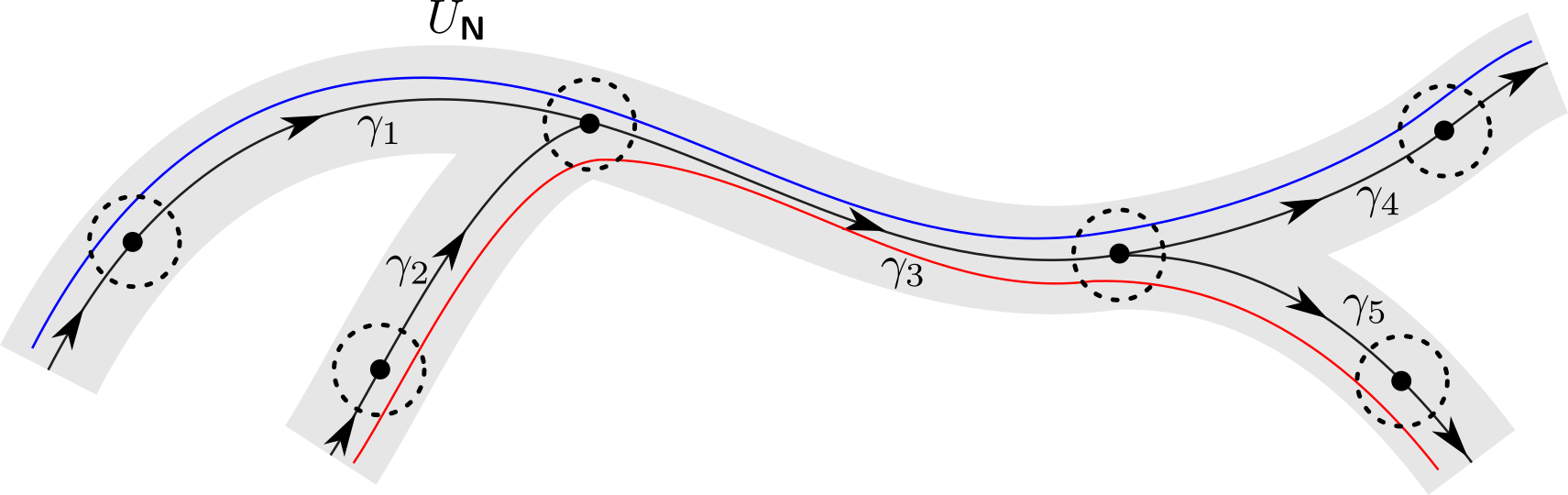}
\caption{Trajectories near a heteroclinic network: The blue trajectory follows the sequence~$134$ while the red trajectory follows~$235$ in the neighborhood $U_\Net$ (shaded) of the network~$\Net$. Neighborhoods of the equilibria in $\Net$ are shown as dashed circles.}
\label{fig:following}
\end{figure}

\begin{rem}\label{rem:maximal-sequence}
It follows from the definition above that $x \in \mathcal{S}_\delta(\qb)$ implies $x \in \mathcal{S}_\delta(\Shft\qb)$. This can be seen by applying a suitable shift to the sequences of times~$t_k, t'_k$. Hence, $\mathcal{S}_\delta(\qb) \subset \mathcal{S}_\delta(\Shft\qb)$. In this sense, the sequence that is realized by a given trajectory near the network is unique only up to shifting. In particular, since $t_k,t'_k \geq 0$ above, a \emph{maximal sequence} realized by a given trajectory can be obtained by checking near which connection this trajectory first enters~$B_\delta(\Net)$ permanently, i.e., it remains in~$B_\delta(\Net)$ for all future times.

Equivalenty, a single trajectory realizes a countable number of sequences on the network, all of which are related by the shift operator~$\Shft$. This number is finite if and only if the maximal sequence realized by this trajectory is preperiodic: 
By the argument above, $x \in \mathcal{S}_\delta(\qb)$ implies $x \in \mathcal{S}_\delta(\Shft^k\qb)$ for all $k \in \N$.
This means that if $x$~realizes the sequence~$\qb$, then $x$~realizes~$\Shft^k \qb$ for all $k \in \N$. 
If~$\qb$ is aperiodic, then $\Shft^k \qb \neq \Shft^j \qb$ for $k \neq j$ and $x$~realizes the infinitely many distinct sequences~$\Shft^k\qb$. 
If on the other hand~$\qb$ is preperiodic, then the set $\{ \Shft^k\qb \mid k \in \N \}$ contains only finitely many distinct sequences, realized by~$x$.
\end{rem}

\begin{defn}
There is \emph{(infinite) switching in positive time} near the network~$\Net$ if in any $\delta$-neighborhood of~$\Net$, every sequence~$\qb$ is realized by at least one trajectory, that is, for any $\delta>0$ and all~$\qb$ we have $\mathcal{S}_\delta({\qb})\neq\emptyset$.
\end{defn}

\begin{rem}
Note that switching in positive time is sometimes also called \emph{forward switching}. Analogously, one can consider \emph{backwards switching}, or even \emph{bi-infinite switching}. In this work we focus exclusively on forward switching and thus refer to it solely as \emph{switching}.

Notions for different levels of switching appear in the literature: Restricting to finite sequences in the above definition naturally leads to the concept of \emph{finite switching}. Considering only particular subsets of finite sequences then gives rise to the even weaker notions of \emph{switching along a node/connection/cycle}, see~\cite{Aguiar2005,CastroLohse2016}, for example.
\end{rem}

Infinite switching has been shown to exist in heteroclinic networks where the linearization of the vector field has non-real eigenvalues for at least one equilibrium in the cycle: 
In~\cite{Aguiar2005}, for instance, the authors study a network between two nodes on a three-dimensional manifold and, under some assumptions, prove the presence of horseshoes in any neighborhood of the network, which implies infinite switching.
Switching has also been considered in the context of homoclinic networks~\cite{doi:10.1080/14689361003769770}.
As a specific instance, \cite{Rodrigues2015a}~provides an example of a homoclinic network (again with non-real eigenvalues) which exhibits infinite switching without the presence of suspended horseshoes---the network itself is asymptotically stable in this case, and thus attracts all initial conditions in some neighborhood.

When restricting to equilibria with real eigenvalues, it was shown recently in~\cite{Castro_2023} that for a large class of networks infinite switching is impossible. 
Even finite switching is restricted:
All finite sequences of a given length~$k$ are realized in a network only up to a maximal bound $k^\mathrm{switch} \in \N$, i.e., if all sequences of length~$k$ are realized, then $k<k^\mathrm{switch}$.

\begin{ex}\label{ex:KirkSilberSwitching}
We illustrate this restriction for the Kirk--Silber network (see Figure~\ref{fig:ks} and Example~\ref{ex:KS}). Consider sequences of length three with~$\gamma_1$ in the middle entry; the four possibilities that are compatible with the network architecture are~$312$, $314$, $512$, $514$.
In their original paper~\cite{VKirk_1994}, Kirk and Silber already showed that for any given parameter configuration only three of these are realized. 
Thus, there is no switching near the Kirk--Silber network, not even along the common connection~$\gamma_1$. 
In the framework of~\cite{Castro_2023}, this corresponds to $k^\mathrm{switch}=3$, that is there is a sequence of length three on the network that is not realized (which one it is depends on the parameters). 
By contrast, any sequence of length $k \leq 2$ is always realized, which implies switching at every node in the network.
\end{ex}

\section{Asymptotic stability of sequences}
\label{sec:Stability}
\noindent
For a given sequence~$\qb$ on a network~$\Net \subset \Rn$, we are interested in how many initial conditions~$x$ lead to trajectories~$\phi_t(x)$ that not only stay close to~$\Net$ but rather approach~$\Net$ as $t\to\infty$. 
From now on, let~$\mu$ denote the Lebesgue measure on $\Rn$.

\subsection{Fragmentary asymptotic stability of sequences}

For a set $A\subset\Rn$ and $x\in\Rn$ write $\dd(A,x)=\inf_{y\in A}{\norm{x-y}}$ for the Euclidean distance between~$A$ and~$x$.
Let
\begin{align}
\mathcal{A}_\delta({\qb}) &:= \tset{x\in\mathcal{S}_\delta({\qb})}{\lim_{t\to\infty}\dd(\Net, \phi_t(x))=0}.
\end{align}
denote points attracted to~$\Net$ along~$\qb$ within a $\delta$-neighborhood. Because of~$\mathcal{S}_\delta(\qb) \subset \mathcal{S}_\delta(\Shft\qb)$ we also have the inclusion $\mathcal{A}_\delta(\qb)\subset\mathcal{A}_\delta(\Shft{\qb})$ and hence, more generally, $\mathcal{A}_\delta(\Shft^n{\qb})\subset\mathcal{A}_\delta(\Shft^m{\qb})$ for $m\geq n$.

Typically, stability is defined for sets which are (forward) invariant with respect to the flow~$\phi_t$.
Our set~$\mathcal{A}_\delta({\qb})$ is not forward invariant: 
The trajectory through a point $x\in\mathcal{A}_\delta({\qb})$ follows $\qb$ by definition. But the same is usually not true if we cut off a finite transient by considering $\phi_t(x)$ for $t>0$ sufficiently large, since $\phi_t(x)$ might be in an entirely different part of the network than the heteroclinic connection that is the first entry of $\qb$. Thus, $\phi_t(x) \not\in \mathcal{A}_\delta(\qb)$, which means $\mathcal{A}_\delta(\qb)$ is not invariant. Applying the shift $\Shft$, however, we can also remove the transient from $\qb$, which means for some $\kappa(t)\in \N$ we have $\phi_t(x) \in \mathcal{A}_\delta(\Shft^{\kappa(t)}\qb)$.

A concrete example for this can be seen in Figure~\ref{fig:following} for $\qb=134\dotsb$.
Take a point~$x$ on the blue trajectory in a $\delta$-neighborhood of~$\alpha[\gamma_1]$, which is in~$\mathcal{A}_\delta(134\dotsb)$. 
In forward time there is a $t>0$ such that $\phi_t(x)$ is in a $\delta$-neighborhood of $\omega[\gamma_1]=\alpha[\gamma_3]$.
This means that \[\phi_t(x)\in\mathcal{A}_\delta(34\dotsb) = \mathcal{A}_\delta(\Shft(134\dotsb)),\] but in general $\phi_t(x)\not\in\mathcal{A}_\delta(134\dotsb)$. In this case we have $\kappa(t)=1$.

This observation together with the desire for an invariant set that collects initial conditions with identical (asymptotic) behavior near the network prompts the next definition.

\begin{defn}
Let~$\qb$ be a sequence on a network~$\Net$ and~$\delta>0$. Define the \emph{asymptotic $\delta$-basin of attraction of~$\qb$} to be
\begin{align}
\mathcal{B}_\delta({\qb}) &:= \set{x\in B_\delta(\Net)}{\exists k \in \N \text{ such that } x \in \mathcal{A}_\delta(\Shft^k\qb)} = \bigcup_{k \in \N} \mathcal{A}_\delta(\Shft^k\qb).
\end{align}
\end{defn}

The asymptotic basin is forward invariant and contains all points that approach~$\Net$ and follow~$\qb$ eventually in the sense that we allow for arbitrary shifts of~$\qb$.
We now define fragmentary asymptotic stability in terms of the asymptotic basin of attraction of a given sequence.

\begin{defn}
Consider a sequence $\qb = q_1q_2q_3\ldots$ on a network~$\Net$. Then~$\qb$ is called \emph{fragmentarily asymptotically stable (f.a.s.)} if for any~$\delta>0$ we have \[\mu(\mathcal{B}_\delta(\qb))>0.\]
\end{defn}

This notion of stability can be refined by considering asymptotic properties along the sequence as follows: The asymptotic basin of attraction decomposes into disjoint sets
\begin{align*}
\mathcal{D}^{(k)}_\delta(\qb):= \mathcal{A}_\delta(\Shft^k\qb) \setminus \mathcal{A}_\delta(\Shft^{k-1}\qb).
\end{align*}
In other words, $\mathcal{B}_\delta(\qb)$~can be written as the disjoint union
\begin{align}
\mathcal{B}_\delta({\qb}) = \bigcup_{k \in \N} \mathcal{D}^{(k)}_\delta(\qb).
\end{align}
Note that $\mu(\mathcal{D}^{(k)}_\delta(\qb))\leq \mu(\mathcal{A}_\delta({\Shft^k\qb})) \leq \mu(\mathcal{B}_\delta({\qb})) \leq \mu(B_\delta(\Net))<\infty$ by inclusion. 
Thus, $\mu(\mathcal{D}^{(k)}_\delta(\qb)) \to 0$ for $k \to \infty$ and we can refine the notion of fragmentary asymptotic stability based on how fast the size of $\mathcal{D}^{(k)}_\delta(\qb)$ decays.

\begin{defn}
Consider an f.a.s.\ sequence $\qb = q_1q_2q_3\ldots$ on a network~$\Net$.
\begin{enumerate}
\item[(i)] The sequence~$\qb$ is called \emph{f.a.s.\ of finite type (f-f.a.s.)} if for any $\delta>0$ there exists a $k_0 \in \N$ such that for any $k>k_0$ we have \[\mu(\mathcal{D}^{(k)}_\delta(\qb))= 0.\]
\item[(ii)] The sequence~$\qb$ is called \emph{f.a.s.\ of infinite type (i-f.a.s.)} if for any $\delta>0$ and all $k \in \N$ we have \[\mu(\mathcal{D}^{(k)}_\delta(\qb))> 0.\]
\end{enumerate}
\end{defn}

\begin{rem}
Here we only define fragmentary asymptotic stability of a particular sequence. In the same spirit though, one may define essential asymptotic stability (e.a.s.) of a sequence~$\qb$, which corresponds to attracting a set that approaches full Lebesgue measure in~$B_\delta(\Net)$ as~$\delta \to 0$; for a general definition of e.a.s.\ for heteroclinic cycles see also~\cite{Melbourne1991}.
\end{rem}

\subsection{Stability of (pre)periodic sequences}

This notion of fragmentary asymptotic stability of sequences on a network~$\Net$ is compatible with the classical notion of fragmentary asymptotic stability from~\cite{Podvigina2012}.
Recall that cycles and omnicycles correspond to periodic sequences.
The following proposition is immediate from the definition and allows us to use the notion of f.a.s.\ for sequences and the corresponding invariant sets interchangeably.

\begin{prop}
If a sequence~$\qb$ is f.a.s., then~$\Gamma(\qb)$ is f.a.s.\ as an invariant set. Conversely, if a heteroclinic (omni)cycle is f.a.s.\ (as an invariant set), then the corresponding periodic heteroclinic sequence is f.a.s.
\end{prop}

We can make a more general statement about the stability of (pre)periodic sequences. 
For any preperiodic sequence there exist $m, n \in \N$ such that $\Shft^{m}\qb = \Shft^{m+n}\qb$ and thus $\mathcal{D}^{m+n}_\delta(\qb) = \emptyset$.
This yields the following conclusion.

\begin{prop}
If a preperiodic sequence~$\qb$ is f.a.s., then~$\qb$ is f-f.a.s.
\end{prop}

\section{Aperiodic sequences}
\label{sec:Aperiodicity}

While periodic sequences correspond to cycles and omnicycles, the main question we address here now is how many aperiodic sequences may be f.a.s.

\begin{prop}\label{prop:CountablyMany}
Suppose~$\qb$ is aperiodic and f.a.s. Then there are countably infinitely many aperiodic sequences that are f.a.s.
\end{prop}

\begin{proof}
If~$\qb$ is f.a.s., then by definition $\mu(\mathcal{B}_\delta(\qb))>0$. Since $\mathcal{B}_\delta(\qb) \subset \mathcal{B}_\delta(\Shft \qb)$, we have $\mu(\mathcal{B}_\delta(\Shft\qb)) \geq \mu(\mathcal{B}_\delta(\qb))>0$. In the same way it follows that
\[\mu(\mathcal{B}_\delta(\Shft^k\qb)) \geq \mu(\mathcal{B}_\delta(\Shft^{k-1}\qb)) \geq \ldots \geq \mu(\mathcal{B}_\delta(\qb))>0.\]
Thus,  $\Shft^{k}\qb$ is f.a.s.\ for any $k \in \N$. Since~$\qb$ is aperiodic, there is a countably infinite set of sequences $\Shft^{k}\qb$ which are distinct and all f.a.s., proving the assertion.
\end{proof}

Before we make our next statement we need the following lemma on possible intersections of the $\delta$-stable sets of two different sequences.

\begin{lem}\label{lem:S-disjoint}
If two sequences $\qb_1$ and $\qb_2$ are not related by the shift $\Shft$, i.e., there is no $k \in \N$ such that $\Shft^k \qb_1 = \qb_2$ or $\Shft^k \qb_2 = \qb_1$, then
$$\mathcal{S}_\delta(\qb_1) \cap \mathcal{S}_\delta(\qb_2) = \emptyset$$
for $\delta>0$ sufficiently small.
\end{lem}

\begin{proof}
This follows from Remark~\ref{rem:maximal-sequence}: For a given $x\in\mathcal{S}_\delta(\qb_1)$ there exists a maximal sequence~$\qb$ such that $x$ follows $\qb$ in $B_\delta(\Net)$. In particular, there is a $k_1\in\N_0$ with $\qb_1=\Shft^{k_1}\qb$. For all other sequences~$\qb_2$ that are realized by this $x$, i.e., $x\in\mathcal{S}_\delta(\qb_2)$, there exists $k_2\in\N_0$ such that $\qb_2=\Shft^{k_2}\qb$. So if $x\in\mathcal{S}_\delta(\qb_1)\cap\mathcal{S}_\delta(\qb_2)$, then $\qb_1$ and $\qb_2$ are related by~$\Shft$.
\end{proof}

As a direct consequence the asymptotic basins of attraction of two sequences that are not related by $\Shft$ are also disjoint. 
With this we can state our next result that---through measure-theoretic arguments---limits the total number of f.a.s.\ sequences that may coexist in a given network.

\begin{prop}\label{prop:NotTooMany}
For any network~$\Net$ and any $\delta>0$, there are at most countably many sequences $\qb$, which are pairwise not related by the shift~$\Shft$ and satisfy $\mu(\mathcal{S}_\delta(\qb))>0$.
\end{prop}

\begin{proof}
Fix $\delta>0$ and a collection of sequences $\mathbf{Q}$ such that any two distinct elements of $\mathbf{Q}$ are not related by~$\Shft$. We consider the family of sets
\[\F:= \{\mathcal{S}_\delta(\qb) \mid \qb \in \mathbf{Q} \text{ and } \mu(\mathcal{S}_\delta(\qb))>0 \},\]
and show that it is countable, regardless of the choice of $\mathbf{Q}$.

Writing $B_r(0) \subset \Rn$ for a ball around the origin with radius $r>0$, we then set
\[\F_{r,k}:= \left\{ \mathcal{S}_\delta(\qb) \in \F \ \Big\vert \ \mu(\mathcal{S}_\delta(\qb) \cap B_r(0))> \frac{1}{k} \right\} \subset \F.\]
For any $\mathcal{S}_\delta(\qb) \in \F$ we can choose $r,k \in \N$ sufficiently large, so that $\mathcal{S}_\delta(\qb) \in \F_{r,k}$, hence
\[\bigcup\limits_{r,k=1}^\infty \F_{r,k} = \F.\]

Now fix $r,k \in \N$ and let $\mathcal{S}_\delta(\qb_1), \mathcal{S}_\delta(\qb_2), \dotsc,\mathcal{S}_\delta(\qb_m) \in \F_{r,k}$. By Lemma~\ref{lem:S-disjoint} these sets are pairwise disjoint. Then
\[\mu(B_r(0)) \geq \mu\left(\bigcup\limits_{j=1}^m (\mathcal{S}_\delta(\qb_j) \cap B_r(0))\right) = \sum_{j=1}^m \mu(\mathcal{S}_\delta(\qb_j) \cap B_r(0)) \geq \frac{m}{k},\]
and therefore $m \leq k \cdot \mu(B_r(0)) < \infty$. Thus, each $\F_{r,k}$ contains only finitely many sets, which means that~$\F$ is a countable union of finite sets, hence countable.

This applies to any choice of $\mathbf{Q}$, which proves the claim.
\end{proof}

We can directly deduce the following statement.

\begin{cor}\label{cor:countable}
For any network~$\Net$ there can be at most countably many f.a.s.\ sequences.
\end{cor}

\begin{proof}
For a sequence $\qb$ to be f.a.s.\ it is necessary that $\mu(\mathcal{S}_\delta(\bar{\qb}))>0$ with $\bar{\qb}:=\Shft^k\qb$ for some $k\in\N$. So having uncountably many f.a.s.\ sequences implies that there are uncountably many sequences~$\bar{\qb}$ with~$\mu(\mathcal{S}_\delta(\bar{\qb}))>0$. This contradicts Proposition~\ref{prop:NotTooMany} because every sequence is related only to countably many others by the shift~$\Shft$.
\end{proof}

These results relate to the notion of \emph{abundant switching}~\cite{Rodrigues2023}.
There, the authors consider a system with infinite switching, but they also investigate the number of trajectories realizing a given sequence---instead of simply showing that there is \emph{at least one} trajectory for every sequence, which suffices for the usual notion of switching. 
Corollary~\ref{cor:countable} means that high levels of switching cannot be achieved by realizing every element in an uncountable collection of sequences through a set of positive measure.

\section{Examples}
\label{sec:Examples}
\noindent
In the following, we discuss whether or not to expect f.a.s.\ aperiodic heteroclinic sequences in concrete examples. 
The first two examples (Section~\ref{sec:Ex1} and~\ref{sec:Ex2}) relate to complexity that arises through potential itineraries a trajectory can take around a network. The third example (Section~\ref{sec:Ex3}) considers complexity of the dynamics in product systems of two homoclinics with different attraction speeds.
While the construction of a concrete example that exhibits an f.a.s.\ aperiodic heteroclinic sequence is beyond the scope of this paper, the systems we discuss give insights how such examples may be constructed.

\subsection{Geometric restrictions on the realized sequences} 
\label{sec:Ex1}

Near the Kirk--Silber network~\cite{VKirk_1994} aperiodic sequences are not realized because it is impossible for trajectories within a small neighborhood to go from one cycle to the other and back; cf.~Example~\ref{ex:KirkSilberSwitching}.
If the symbols~$\mathrm{A}$ and~$\mathrm{B}$ correspond to the constituent cycles (as in Example~\ref{ex:KS}), this means that $\mathrm{ABA}$ and $\mathrm{BAB}$ cannot appear in any realized sequence. 
In fact, results in~\cite{DissLohse} imply that the only sequences realized in the Kirk--Silber system are of the form $\mathrm{A}^k\mathrm{BBB}\dotsb$ with arbitrary $k \in \N$ (or with the roles of~$\mathrm{A}$ and $\mathrm{B}$ reversed, depending on the choice of parameters). This fully characterizes the set of sequences that are realized in the Kirk--Silber system. In particular, all realized sequences are preperiodic and thus no aperiodic sequence is realized (much less f.a.s.) here.

The reason for these restricted dynamics are geometric obstructions caused by the low-dimensional space ($\R^4$) containing the network. Similar limitations exist for networks containing a Kirk--Silber subnetwork \cite{AGUIAR20111475} or simply a common connection between two cycles satisfying an assumption on the global maps which retains the essence of the Kirk--Silber situation, see~\cite{CastroLohse2016}.

Podvigina~\cite{Podvigina2021} discusses a network in $\R^6$ where these restrictions do not apply. Its graph can be created from the Kirk--Silber network by adding two new nodes, thus replacing both $[\xi_3 \to \xi_1]$ and $[\xi_4 \to \xi_1]$ by $[\xi_3 \to \xi_3' \to \xi_1]$ and $[\xi_4 \to \xi_4' \to \xi_1]$, respectively\footnote{We use a different labeling for the nodes than Podvigina to highlight the similarity to the Kirk--Silber network.}; see Figure~\ref{fig:pod}.
By slight abuse of terminology we label the cycles $[\xi_1 \to \xi_2 \to \xi_3 \to \xi_3' \to \xi_1]$ and $[\xi_1 \to \xi_2 \to \xi_4 \to \xi_4' \to \xi_1]$ as~$\mathrm{A}$ and~$\mathrm{B}$, respectively. Podvigina shows that under suitable assumptions the periodic sequence (an {omnicycle}) defined through $\overline{\mathrm{A}\mathrm{B}}$ is f.a.s.\ (which she calls \emph{trail-stable}).

More generally, as pointed out in Section~\ref{sec:Sequences}, results in~\cite{Castro_2023} show that having only real eigenvalues prevents infinite switching and restricts the set of finite sequences that may be realized. 
However, this does not necessarily mean that no aperiodic sequences can be realized (and potentially be f.a.s.) in such systems. 
On the one hand, Postlethwaite and Dawes~\cite{Postlethwaite_2005} derived restrictions for possible periodic ways to approach a heteroclinic network with strong stability properties.
While they observe irregular cycling dynamics in the sense that all periodic sequences on the network they consider are unstable, the stability of individual aperiodic sequences remains unclear.
On the other hand, Podvigina  presents numerical evidence hinting at potentially stable aperiodic dynamics in~\cite{Podvigina2021}, but at the same time states that, numerically, aperiodic sequences are practically indistinguishable from long periodic ones.

\begin{figure}
\centering
	\begin{tikzpicture}[>=latex,thick,scale=0.8]
	\coordinate (1) at (-3,0);
	\coordinate (2) at (3,0);
	\coordinate (3) at (1.5,2);
	\coordinate (4) at (-1.5,2);
	\coordinate (5) at (1.5,-2);
	\coordinate (6) at (-1.5,-2);
	\filldraw[black] (1) circle (2.4pt) node[left] {$\xi_1$};
	\filldraw[black] (2) circle (2.4pt) node[right] {$\xi_2$};
	\filldraw[black] (3) circle (2.4pt) node[right] {$\xi_3$};
	\filldraw[black] (4) circle (2.4pt) node[left] {$\xi_3'$};
	\filldraw[black] (5) circle (2.4pt) node[right] {$\xi_4$};
	\filldraw[black] (6) circle (2.4pt) node[left] {$\xi_4'$};
	\path (1) edge[middlearrow={>}{left}{}] (2);
	\path (2) edge[middlearrow={>}{left}{}] (3);
	\path (3) edge[middlearrow={>}{left}{}] (4);
	\path (4) edge[middlearrow={>}{left}{}] (1);
	\path (2) edge[middlearrow={>}{left}{}] (5);
	\path (5) edge[middlearrow={>}{left}{}] (6);
	\path (6) edge[middlearrow={>}{left}{}] (1);
	\node at (0,1.0) {$\mathrm{A}$};
	\node at (0,-1.0) {$\mathrm{B}$};
	\end{tikzpicture}
\caption{Podvigina's network from~\cite{Podvigina2021} with an f.a.s.\ periodic sequence $\overline{\mathrm{A}\mathrm{B}}$.}
\label{fig:pod}
\end{figure}
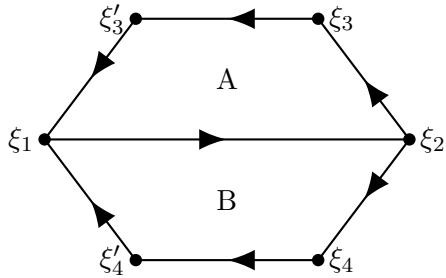

The authors of~\cite{Postlethwaite2022} offer similar results for a network in the context of the Rock-Paper-Scissors-Lizard-Spock game: They prove fragmentary asymptotic stability for a number of periodic sequences, but also encounter regions in parameter space where they cannot find stable periodic behavior. Their numerical results indicate stable aperiodic dynamics, but a proof of aperiodicity, in the sense that there is an aperiodic f.a.s.\ sequence, is lacking.

\subsection{Aperiodic sequences in systems with infinite switching} 
\label{sec:Ex2}

The examples discussed above have obstructions to the emergence of aperiodic sequences in general. However, even if switching is present, the stability of an individual (aperiodic) sequence needs to be clarified.
Infinite switching has been observed in a number of systems with at least one equilibrium where the linearization of the vector field has non-real eigenvalues. 
From this perspective, usually the number of initial conditions realizing a given sequence is of little concern. 
Since infinite switching often occurs on a Cantor set of measure zero, it does not necessarily exclude the possibility of stable aperiodic dynamics in a system.

However, the absence of f.a.s.\ aperiodic sequences is a reasonable conjecture in such systems.
Their analysis typically does not suggest that some sequences have qualitatively different stability properties than others.
At the same time, if all possible sequences on a network are realized, then Corollary~\ref{cor:countable} implies that they cannot all be f.a.s.\ simultaneously. 

This reasoning applies for example to the system studied in~\cite{Rodrigues2023}, where the authors consider a heteroclinic network consisting of two saddle-foci~$P_1$ and~$P_2$ with two connections from~$P_1$ to~$P_2$ and an arbitrarily large (but finite) number of connections from~$P_2$ to~$P_1$. They show that there is infinite switching near this network and every sequence is realized by an infinite number of trajectories. Since all of the connections from~$P_2$ to~$P_1$ have the same properties (as far as the proof for switching is concerned), it seems reasonable to conclude that all possible sequences display the same level of stability, thus suggesting that none of them is~f.a.s.

\subsection{Two homoclinic loops} 
\label{sec:Ex3}

In our third example, we discuss a different potential source for the emergence of f.a.s.\ aperiodic sequences in a product system of two stable homoclinic cycles with distinct convergence speeds.
Specifically, we study a toy example of two homoclinic loops to derive heuristic arguments for the presence or absence of f.a.s.\ aperiodic sequences in a certain class of systems.
The main insight is that if the equilibrium is hyperbolic, one should expect preperiodic rather than aperiodic sequences.
This indicates possible obstructions for the emergence of f.a.s.\ aperiodic sequences and suggests future work with a focus on two homoclinic loops from a nonhyperbolic equilibrium.

Consider a product system in $\R^2 \times \R^2$ with a hyperbolic saddle equilibrium at the origin and two attracting homoclinic trajectories (one in each factor of the system) converging to the origin in forward and backward time. 
This produces a homoclinic network~$\Net$ comprised of two cycles, say~$\mathrm{A}$ and~$\mathrm{B}$.
Typically, this network is attracting in such a product system (rather than the product of the cycles)~\cite{AshwinField2005}. 
This means trajectories approaching this attractor can be assigned a sequence in these two symbols corresponding to successive transitions along one cycle or the other.
In contrast to~\cite{ArgawalRodriguesField2011}, where the absence of finite switching was proved for such a system, we here focus on the possibility to obtain aperiodic sequences as points approach~$\Net$.

F.a.s.\ aperiodic dynamics are possible in such a system if the number of transitions along say $\mathrm{A}$ strictly increases between successive transitions along~$\mathrm{B}$.
Thus, our goal is to estimate an asymptotic relation between the number of turns it takes around each of the homoclinic cycles within a given time, depending on the rates of contraction for the two cycles.
This then provides a hint as to whether typical sequences encoding such trajectories are periodic or not.

We start by deriving a simplified return map for a system as outlined above, which we now write as
\begin{align}\label{sys-2hom}
(\dot x_j, \dot y_j)= f_j(x_j,y_j)
\end{align}
for~$j\in\{\mathrm{A}, \mathrm{B}\}$.
If two homoclinic cycles exist in the $(x_j,y_j)$-subspaces, we may assume the linearization of the vector field at the origin to be of the form $\diag(-c_A,e_A,-c_B,e_B)$ with $c_j, e_j >0$.
In particular, we assume that stability near the origin is governed by the linear terms.

Stability of each homoclinic cycle can be computed using the Poincar\'e return map for cycle $j\in\{\mathrm{A}, \mathrm{B}\}$.
First, consider the local transition through an $\varepsilon$-neighborhood in one of the two $(x_j,y_j)$-subspaces of system~\eqref{sys-2hom}. 
A standard calculation using the solution of the linearized system shows that with (sufficiently small) inital conditions $x_j(0)=\delta$ and $y_j(0)=\varepsilon$ we obtain the transition time $T=T(\delta)=\tilde{c}\ln(\delta)$, i.e., $y_j(T)=\varepsilon$, where $\tilde{c}<0$ is some constant. 
Thus, $x(T)=\hat{c}\delta^{\nu_j}$, where $\nu_j=\frac{c_j}{e_j}$, with another constant $\hat{c}>0$. 
This yields the local map $\delta \mapsto \hat{c}\delta^{\nu_j}$ for the transition through a neighborhood of the origin.
Second, one can use tubular flow arguments to derive a global map to capture the flow along the connection.
While the composition of the local and global maps yields the Poincar\'e return map, it is well-known that the asymptotic stability of the cycle only depends on the local properties:
The cycle is asymptotically stable if $c_j>e_j$ or, equivalently, if $\nu_j>1$. 
In particular, as trajectories approach the homoclinic cycle, more and more time is spent within the $\varepsilon$-neighborhood of the origin.

{\allowdisplaybreaks
To approximate the time it takes a nearby trajectory to complete a number of turns around one of the two cycles, we ignore the global map entirely, drop the constant~$\hat{c}$ and consider iterations of the map $\delta \mapsto \delta^{\nu_j}$. 
With this simplification, the approximate time~$T^n(\delta)$ for~$n$ turns around a cycle by a trajectory with initial condition~$\delta$ is
\begin{align*}
T^n(\delta)&=T(\delta) + T(\delta^{\nu_j}) + T(\delta^{\nu_j^2}) \ldots + T(\delta^{\nu_j^{n-1}})\\
&=\tilde{c}  \sum\limits_{k=0}^{n-1}\ln{\delta^{\nu_j^k}}
=\tilde{c}\ln{\delta}  \sum\limits_{k=0}^{n-1}\nu_j^k
=\tilde{c}\ln{\delta}  \left(\frac{1-\nu_j^n}{1-\nu_j}\right).
\end{align*}
Thus, given some $T>0$ we can solve the equation
\[T=\tilde{c}\ln{\delta}  \left(\frac{1-\nu_j^n}{1-\nu_j}\right)\]
for $n$ to get an estimate for the number of turns a trajectory with initial condition $x_j(0)=\delta$ makes around a cycle~$j\in\{\mathrm{A}, \mathrm{B}\}$ in a time interval of length~$T$.
}

This allows us to estimate the number of transitions along each cycle in a given time~$T$.
Ignoring the constants, we have
\[\frac{1-\nu_\mathrm{A}^n}{1-\nu_\mathrm{A}} = T = \frac{1-\nu_\mathrm{B}^m}{1-\nu_\mathrm{B}}\]
which we can solve for $n=n(m)$ or for $m=m(n)$ to find a relation between the numbers of turns around each cycle. We obtain
\[n=n(m)=\frac{1}{\ln{\nu_\mathrm{A}}}\ln\left( 1- \frac{1-\nu_\mathrm{A}}{1-\nu_\mathrm{B}} (1- \nu_\mathrm{B}^m) \right),\]
which implies that the relation we were looking for is (asymptotically) linear. 
Despite the simplifications, this indicates that sequences close to the homoclinic loops are predominantly preperiodic rather than aperiodic.

This suggests that---in the context of a system with two homoclinic loops---hyperbolicity of the equilibrium can be an obstruction to the emergence of f.a.s.\ aperiodic sequences.
At the same time, it suggests that f.a.s.\ aperiodic sequences may appear as a bifurcating case when the equilibrium is nonhyperbolic.
We discuss this further below.

\section{Discussion}
\label{sec:discussion}
\noindent
The main take-home messages from this note are that (i)~if there is an f.a.s.\ aperiodic heteroclinic sequence, then there are (countably) infinitely many and (ii)~there cannot be too many (in particular uncountably many) f.a.s.\ aperiodic heteroclinic sequences.
We conclude by discussing some open questions, whose answers are beyond the scope of this note.

First, are all f-f.a.s.\ sequences preperiodic? The stability of preperiodic heteroclinic sequences can be determined by looking at finite products of transition matrices as shown in~\cite{Podvigina2021}. 
The converse would fully classify the set of f-f.a.s.\ sequences and simplify the stability calculations to determining products of matrices.
More generally, is it possible to obtain similar conditions for some infinite product of transition matrices and/or the iterated function system they represent? Since for any preperiodic sequence~$\qb$ we not only have $\mu(\mathcal{D}^{(k)}_\delta(\qb)) =0$, but even $\mathcal{D}^{(k)}_\delta(\qb) = \emptyset$ it is in this context also worth asking whether there exist sequences $\qb$ such that $\mathcal{D}^{(k)}_\delta(\qb)$ is a non-empty set of measure zero. This would constitute yet another intermediate level of stability.

Second, a concrete example of a system with an f.a.s.\ aperiodic sequence is still outstanding.
While systems that support switching show aperiodic dynamics, we do not expect many (if any) aperiodic sequences to be f.a.s.; cf.~Sections~\ref{sec:Ex1} and~\ref{sec:Ex2}.
Potentially the most interesting route would be to consider a product system as discussed in Section~\ref{sec:Ex3} where the two cycles have qualitatively different attraction rates.
More specifically, attraction to~$\mathrm{A}$ may be exponential in the corresponding invariant subspace and governed by the linearization at the equilibrium while attraction to~$\mathrm{B}$ is subexponential (for example, if the equilibrium was nonhyperbolic).
In this case, one may expect that a typical trajectory converges to an aperiodic sequence of the type~$\mathrm{A}\mathrm{B}^{j_1}\mathrm{A}\mathrm{B}^{j_2}\mathrm{A}\dotsb$ with $j_1<j_2<j_3<\dotsb$.

Third, what happens under time-reversal? 
Asymptotic stability relates to convergence as time increases. 
Thus, it is natural to ask what the asymptotic dynamics are as time is reversed and whether it relates to the network itself.
A trajectory through $x\in \Wu(\xi_p)\cap \mathcal{A}_\delta(\qb)$ approaches~$\Net$ in forward time (along~$\qb$) and in backward time~(to $\xi_p$);
this relates to heteroclinic connections of higher depth; see, for example,~\cite{Ashwin1999,ChawanyaAshwin2010}.
More generally, given an aperiodic f.a.s.\ sequence~$\qb$, which (finite) sequences can we put in front of~$\qb$ such that the resulting sequence is also f.a.s.?
Finally, can we have a bi-directional heteroclinic sequence in the sense that there are  trajectories converging to a network in forward and backward time, and if so, are they f.a.s.?

\bigbreak
\noindent\textbf{Acknowledgements.} We are grateful to P.~Ashwin, S.~Castro, C.~Postlethwaite, A.~Rodrigues, and A.~Rucklidge for helpful discussions on the content of this paper. We also thank an anonymous reviewer for their comments and questions that helped improve the exposition.


\bibliographystyle{unsrt}
\def\urlprefix{}
\def\url#1{}

\bibliography{ref} 

\end{document}